\documentclass[12pt]{article}

\usepackage{ABT}
\usepackage{xspace}
\usepackage{mathtools}

\renewcommand{\subseteq}{\subset}

\newcommand{\genskelish}{local\xspace}
\newcommand{\homogeneity}{weak homogeneity\xspace}
\newcommand{\homogeneous}{weakly homogeneous\xspace}
\newcommand{\card}[1]{\|#1\|}
\newcommand{\R}{F} 
\newcommand{\RSC}{\R_{\mathrm{SC}}}
\newcommand{\AlgR}{F}
\newcommand{\K}{\mathbf{K}}
\newcommand{\Ksc}{\K_{\mathrm{SC}}}

\newcommand{\Ks}[1]{\K^{\subset#1}}
\newcommand{\Lsc}{\mathcal{L}_{\mathrm{sc}}}

\newcommand{\Ssc}{\Sigma_{\mathrm{sc}}}
\newcommand{\Thrsc}{\mathit{Th}_{\R}} 
\newcommand{\Thr}{\mathit{Th}_{\R}} 
\newcommand{\Fraisse}{Fra\"\i ss\'e\xspace}

\DeclareMathOperator{\Aut}{Aut}
\DeclareMathOperator{\Mod}{Mod}
\DeclareMathOperator{\Str}{Str}

\newcommand{\PhiGI}[1]{\Phi_{\mathrm{GI}}^{#1}}

\newcommand{\fin}[1]{[#1]^{fin}}
\newcommand{\simplex}[1]{\Delta_#1}
\newcommand{\SC}{S} 

\usepackage{color}

\newcommand{\zo}{0-1 law\xspace}
\newcommand{\zos}{0-1 laws\xspace}
\newcommand{\Zo}{0-1 Law\xspace}

\newcommand{\fr}[2]{#2^{(#1)}}
\newcommand{\f}[2]{#1^{(#2)}}

\begin{document}
\title{The infinite random simplicial complex}
\author{Andrew Brooke-Taylor\thanks{Supported by 
the Heilbronn Institute for Mathematical Research at the University of Bristol 
and at Kobe University by 
a JSPS Postdoctoral Fellowship for Foreign Researchers and 
JSPS Grant-in-Aid no. 23 01765.}\ \ and Damiano Testa\thanks{Damiano Testa, Mathematics Institute, University of Warwick, Coventry, CV4 7AL, United Kingdom}}
\date{\today}
\maketitle

\section*{Abstract}
We study the \Fraisse limit of the class of all finite simplicial complexes.
Whilst the natural model-theoretic setting for this class
uses an infinite language,
a range of 
results associated with \Fraisse limits of structures for finite languages
carry across to this important example. 
We introduce the notion of a \emph{local class}, with the class of
finite simplicial complexes as an archetypal example, and in this 
general context
prove the existence of a 0-1 law and other basic model-theoretic results.
Constraining to the case where all relations are symmetric, we
show that every direct limit of finite groups, and every metrizable profinite
group, appears as a subgroup of the automorphism group of the \Fraisse limit.
Finally, for the specific case of simplicial complexes,
we show that its geometric realisation is topologically surprisingly simple: 
despite the combinatorial complexity of the \Fraisse limit,
its geometric realisation is homeomorphic to the infinite simplex.

\section{Introduction}

Simplicial complexes are important objects of study in a 
variety of areas of mathematics, including 
algebraic topology and combinatorics.
Much attention has been paid to random $d$-dimensional simplicial complexes
with complete $(d-1)$-skeleton, for fixed $d$.
Indeed, Rado's original random graph paper \cite{Rad:UGUF} shows that there
is a universal countable $d$-dimensional simplicial complex. 
Blass and Harary \cite{BlH:PAAGC} have shown in this context that a 0-1 law
holds, and more 
recently questions about homology of random $d$-dimensional simplicial complexes
have come to the fore, such as in 
\cite{LiM:HCR2C} and 
\cite{MeW:HCRkdC}.

However, all of this leaves open the analogous questions for random simplicial
complexes with no such dimension restriction. 
In the present paper we address this topic.
Rather than the uniform probability
measure on the set of $n$ vertex
simplicial complexes, we consider what is arguably a more natural measure
for this context: the measure induced by building up simplicial
complexes inductively by dimension, 
tossing a fair coin for each potential $n$-simplex (with
its full $(n-1)$-skeleton already in the complex) 
to determine whether it will be in the final simplicial complex.
We show in Theorem~\ref{01Law} that with this probability measure the
class of all finite simplicial complexes bears a 0-1 law.

The proof of Theorem~\ref{01Law} is actually a relatively
simple modification of known
techniques from the theory of \emph{\Fraisse classes}, and it is in this
setting that we frame our results.  
We define the notion of a \emph{local class},
a notion which resembles and is closely related to Oberschelp's
notion of a parametric class \cite{Obe:MSNTWI},
and which will include
the class of finite simplicial complexes as an example.
We then 
isolate a property common to local classes and parametric classes,
the \emph{Adoptive Property},
and use this property to demonstrate the existence of a \zo for 
local classes.
The Adoptive Property 
is also similar to but not quite subsumed by an extant notion 
in the literature, that
of \emph{admitting substitutions} due to Koponen~\cite{Kop:APEPRlCS},
and the arguments used are correspondingly similar.

Integral to this approach is the countable structure known as the
\emph{\Fraisse limit} of the class of structures in question, 
which in the simplicial 
complex case we dub the \emph{infinite random simplicial complex} $\RSC$.
The automorphism groups of \Fraisse limits are an area of extensive 
research (see for example \cite{MTe:SSAG, Tru:GCUG, Tru:43}), 
and as a nod to this we prove in Section~\ref{algebra}
some basic results about the range of subgroups of $\Aut(\RSC)$.
Again our results hold in broader generality, applying to 
the automorphism group of any local class 
for which all relations are symmetric.

Linial and Meshulam \cite{LiM:HCR2C} conclude by opining that
``\ldots further study of topological properties of random complexes will 
prove both interesting and useful.''
In Section~\ref{topology} we 
prove 
a surprising result
about the topology of $\RSC$: despite its high combinatorial complexity, 
$\RSC$ has a topologically very simple geometric realisation, 
homeomorphic to the infinite simplex.
This naturally raises the question of whether properties such as contractibility
are held by almost all finite simplicial complexes: a negative answer would
yield a proof that the property in question is not first order definable.

\section{Preliminaries and definitions}\label{prelims}
%
We begin by fixing our
(mostly standard) model-theoretic notation;
see for example \cite[\S1]{EbF:FMT} or \cite[\S1]{Hod:MT} 
for further details.
Let $\Sigma = \{R_n\st n\in\N\}$ be a relational signature (so, each $R_n$ is
a relation symbol for our language).
We write $\Sigma$-structures as 
$M=\langle |M|, R_0^M, R_1^M,\ldots\rangle$, where 
$|M|$ is the underlying set of $M$ and
$R_n^M$ is the interpretation of $R_n$ in $M$: that is,
if the arity of $R_n$ is $a_n$, then $R_n^M$ is the set of $a_n$-tuples of
elements of $|M|$ for which the relation $R_n$ is said to hold in $M$.
For a $\Sigma$-structure $M$ and a first
order sentence $\phi$ over $\Sigma$
the notation $M\sat\phi$ means that $\phi$ holds of $M$;
similarly, for a set $\Phi$ of first order sentences, $M\sat\Phi$
means that every $\phi\in\Phi$ holds of $M$.
We use $\card{M}$ to denote the cardinality of $M$, that is, the cardinality of the underlying set
$|M|$ of $M$.
For any set $X$ we denote by $[X]^m$ the set of $m$-element subsets of $X$,
and by $\fin{X}$ the set of finite subsets of $X$.
We use $X\subset Y$ to denote that $X$ is a (not necessarily proper) subset
of $Y$.

We shall be primarily interested in classes of finite structures
(in particular, finite simplicial complexes) up to isomorphism; 
for this, it will be convenient to assume that these structures have subsets
of $\N$ as their underlying sets.
Thus, by ``class of structures'' we shall mean a class $\K$ of finite 
structures $M$ with $|M|\subset\N$.
Concomitantly, we use
$\Mod\Phi$ to denote the class of finite $\Sigma$-structures $M$ 
such that $M\sat\Phi$ and $|M|\subset\N$, and similarly 
we denote by $\Str\Sigma$ the class of all finite $\Sigma$-structures $M$ with
$|M|\subset\N$.
Identity of formulas is denoted by $\equiv$, not to be confused with
$=$ (which is used within formulas); thus for example,
$\phi\equiv R(x_1,\ldots,x_n)$ means that $\phi$ is the formula
$R(x_1,\ldots,x_n)$.

It will be important to distinguish between the following two notions of ``substructure'',
as both have a role to play in what is to follow.
\begin{defn}
Suppose $B$ is a structure for a relational signature $\Sigma$.
A \emph{substructure} $A$ of $B$ is the induced substructure on a subset of $B$.
That is, $A$ is a substructure of $B$ if $|A|\subseteq|B|$ and for every
relation symbol $R$ in $\Sigma$ and 
tuple $\mathbf{a}$ from $A$ of length the arity of $R$, 
$R(\mathbf{a})$ holds in $A$ if and only if $R(\mathbf{a})$ holds in $B$.
A \emph{subobject} $B'$ of $B$ is a
$\Sigma$-structure $B'$ such that $|B'|\subseteq|B|$ and the inclusion
map is a $\Sigma$-homomorphism.  
That is, for every 
relation symbol $R$ in $\Sigma$ and
tuple $\mathbf{b}$ from $B'$ of length the arity of $R$, if 
$R(\mathbf{b})$ holds in $B'$ then 
$R(\mathbf{b})$ holds in $B$, but not necessarily conversely.
\end{defn}
If $X\subset|B|$, we shall sometimes write $B\restr X$ to mean the substructure of $B$
with underlying set $X$.
Koponen~\cite{Kop:APEPRlCS} refers to subobjects as 
\emph{weak substructures}, but we shall stick with this less verbose 
terminology from category theory.

The main thrust of our results will be that techniques from the finite
signature case can be made to work for suitably ``locally finite'' 
theories over infinite signatures.  
For this, the following two definitions will be crucial.
\begin{defn}
A relational signature $\Sigma$ is \emph{locally finite} if for every 
$n\in\N$ there are only finitely many $n$-ary relation symbols in $\Sigma$.
\end{defn}
Local finiteness of the signature will not suffice on its own, as
high arity relations with repeated entries can emulate low arity relations.
To overcome this, we employ the following.
\begin{defn}\label{GIdefn}
For any signature $\Sigma$,
we denote by $\PhiGI{\Sigma}$
the set of axioms fitting the following schema:
\begin{description}
\item[General Irreflexivity (GI):]
For all natural numbers $n\geq 2$, all $R$ in $\Sigma$ of arity $n$, 
and all distinct $i,j \in \{1 , \ldots , n\}$,
\[
\forall x_1\cdots\forall x_{n} \;\; 
\Bigl( R(x_1,\ldots,x_n) \longrightarrow (x_i\neq x_j) \Bigr) .
\]
\end{description}
\end{defn}

Also important will be the following kind of subobject,
as they capture local properties.
\begin{defn}
Let $\Sigma$ be a relational signature, let 
$A$ be an $\Sigma$-structure, and let $n\geq 1$ be a natural number.
The \emph{$n$-frame} $\f{A}{n}$ of $A$ is the subobject of $A$ with
the same underlying set $|A|$, the same interpretations $R^A$ of the
$m$-ary relations $R$ in $\Sigma$ for all $m\leq n$, 
but the empty interpretation for all higher arity relations.
\end{defn}
We put parentheses in the superscript for an $n$-frame to distinguish from
the standard indexing convention for skeleta of simplicial complexes,
and to emphasise the point of view of $\f{A}{n}$ as a kind of
$n$-th order approximation to~$A$.

We immediately make two simple but important observations.
\begin{lemma}\label{Onnnframe}
If $\K\subset\Mod(\PhiGI{\Sigma})$ for a locally finite relational signature 
$\Sigma$, then
for any $A\in\K$ of cardinality $n\in\N$, $A$ is equal to its $n$-frame
$\fr{n}{A}$.\hfill\qedsymbol
\end{lemma}
Thus, the $n$-frame of a structure $B$ 
can be thought of as being built up in a natural way
from the cardinality $n$ substructures of $B$.
\begin{lemma}\label{fincardn}
If $\K\subset\Mod(\PhiGI{\Sigma})$ for a locally finite relational signature 
$\Sigma$, then
then for every $n\in\N$, there are up to isomorphism only
finitely many structures in $\K$ of cardinality $n$.\hfill\qedsymbol
\end{lemma}

\subsection{\Fraisse limits}
We recall the basic theory of \Fraisse limits;
see for example \cite[Chapter~7]{Hod:MT} for more details.
Let $\Sigma$ be a countable signature 
(we shall give the explicit signature for simplicial complexes below).
A countably infinite class $\K$ of finitely generated $\Sigma$-structures 
is said to be a
\emph{\Fraisse class} if it satisfies the following three properties.
\begin{description}
\item[Hereditary Property, HP.] For any $B\in\mathbf{K}$ and any
finitely generated
substructure $A$ of $B$, $A$ is also in $\mathbf{K}$. 
\item[Joint Embedding Property, JEP.] For any $B, C\in\mathbf{K}$, there is
a $D\in\mathbf{K}$ such that both $B$ and $C$ embed into $D$.
\item[Amalgamation Property, AP.] 
Suppose that $A, B, C\in\mathbf{K}$ with embeddings 
$f_B \colon A\to B$ and $f_C \colon A\to C$.
Then there are $D\in\mathbf{K}$ and embeddings 
$g_B \colon B \to D$ and $g_C \colon C \to D$ such that 
$g_B\circ f_B=g_C\circ f_C$.
\end{description}
Note that here \emph{substructure} means induced substructure, and an
\emph{embedding} is an injective homomorphism also preserving the negations of
the relations (thus, an isomorphism to its image).
Also note that for relational languages --- that is, those without
function or constant symbols ---
finitely generated simply means finite. 
For any \Fraisse class $\K$,
there is a countable $\Sigma$-structure $F$
such that 
$\K$ is the class of all 
finitely generated $\Sigma$-structures (up to isomorphism) that can be embedded
into $F$ ($\K$ is the \emph{age} of $F$)
and $F$ has the following homogeneity property, which,
following Hodges \cite[Section~7.1]{Hod:MT}, 
we refer to  as \emph{weak homogeneity}.
\begin{defn}\label{wkhom}
A $\Sigma$-structure $F$ is \emph{\homogeneous} if for any
finitely generated $\Sigma$-structure $B$ that embeds into $F$, 
any substructure $A$ of $B$, 
and any embedding $f$ of $A$ into $F$, there is an embedding $g$ of
$B$ into $F$ extending~$f$.
\end{defn}
Weak homogeneity is sometimes also referred to as the
\emph{extension property}.
Any two countable \homogeneous structures of the same age are 
isomorphic (see for example \cite[Lemma~7.1.4]{Hod:MT}),
so such a structure $F$ is unique up to isomorphism.
This $F$ is known as the \emph{\Fraisse limit of $\K$}.

A well-known example of a \Fraisse limit is the countable random graph,
or Rado graph, 
which is the \Fraisse limit of the class of finite graphs.  
However, there are many
other kinds of structures to which the theory can be applied, such as
finite groups (using Neumann's permutation products 
\cite{Neu:PPG} for AP; the \Fraisse limit group was studied by
Hall~\cite{Hall:SCLFG})
and finite rational metric spaces (the completion of the \Fraisse limit of which is the well-known 
\emph{Urysohn space} \cite{Ury:SEMU}).

The elements of a \Fraisse class clearly need only be given up to 
isomorphism.  For concreteness, we shall always assume that any \Fraisse
class $\K$ only contains members $M$ with underlying set a subset of
$\N$ (which we take to include 0), 
and that $\K$ is closed under isomorphism of such structures.

\subsection{Simplicial complexes}\label{PrelimsSC}

Recall that a \emph{simplicial complex} on a set $V$  is 
a subset $\Delta $ of 
$\fin{V}$ (the set of finite subsets or \emph{simplices} of $V$)
which is closed under inclusion: if $x$ 
is in $\Delta$ and $y \subset x$ then $y\in \Delta$.
We refer to those simplices $x$ in $\Delta$ as \emph{faces} 
of $\Delta$.
To couch simplicial complexes in a model-theoretic framework,
we make the following definitions.
\begin{defn}
The \emph{signature of simplicial complexes} $\Ssc$ is the
set $\Ssc=\{S_i\st i\in\N\}$, where for each $i\in\N$,
$S_i$ is an $i+1$-ary relation symbol.
The \emph{language of simplicial complexes} $\Lsc$ is the first order language
comprising formulae built from $\Ssc$ 
in the usual way.
\end{defn}

\begin{defn}\label{SCdefn}
A \emph{simplicial complex} is formally a $\Ssc$ structure satisfying 
$\Phi^{\Ssc}_{\text{GI}}$ and all
axioms fitting either of the following two further schemata:
\begin{description}
\item[Symmetry:] For every positive $n$ and 
every permutation $\sigma$ of the set $\{0,\ldots,n\}$,
\[
\forall x_0 \cdots \forall x_n \;\; \Bigl( S_n(x_0,\ldots,x_n)\longrightarrow 
S_n(x_{\sigma(0)},\ldots,x_{\sigma(n)}) \Bigr).
\]
\item[Subset Closure:] For every positive $n$,
\[
\forall x_0 \cdots \forall x_n (S_n(x_0,x_1,\ldots,x_n)
\longrightarrow S_{n-1}(x_0,\ldots,x_{n-1})).
\]
\end{description}
\end{defn}
General Irreflexivity and Symmetry allow us to encode the $n+1$-element
sets of a subset of $V^{fin}$ by (all of) the ordered $n+1$-tuples of those
elements, and then the Subset Closure axiom schema describes simplicial 
complexes in this context.

Clearly the class of finite simplicial complexes satisfies 
the Hereditary Property, 
the Amalgamation Property and the Joint Embedding Property.
Thus, we may make the following definition.

\begin{defn}
The \emph{infinite random simplicial complex} $\RSC$ is the 
\Fraisse limit of the class of finite simplicial complexes.
\end{defn}

The $k$-frame $\fr{k}{\Delta}$ of a simplicial complex $\Delta$ is an 
important notion in the study of simplicial complexes, called the 
$(k-1)$-skeleton of $\Delta$ and traditionally denoted $\Delta^{k-1}$.
This shift of 1 from the cardinality of the faces to the index makes sense as
the dimension of the geometric realisation
of $\Delta$ (see Section~\ref{topology} below),
but would make little sense in the general model-theoretic setting we use,
so we have elected to introduce the new notation and terminology of
$k$-frames.

\section{Local classes}
A major feature of the theory surrounding \Fraisse limits is the existence of
\zos.
Glebskii, Kogan, Liogon'kii and Talanov~\cite{GKLT:RDR} and
Fagin~\cite{Fag:PFM} showed that the \Fraisse class of all structures for
a finite language bears a \zo.
Oberschelp~\cite{Obe:A01LC} generalised this to so-called 
\emph{parametric classes}
(see below for a definition).
More recently, 
Koponen~\cite{Kop:APEPRlCS} has generalised further,
undertaking a detailed study of \zos
based on extension axioms (which we too shall employ).
A common feature of these results is that the underlying
signature is generally taken
to be finite (although Koponen's framework in 
\cite{Kop:APEPRlCS} frequently admits the
infinite case).  Indeed this is reasonable ---
in general, classes of structures for an infinite signature will have
infinitely many structures of each finite cardinality, ruling out
numerical asymptotic probability calculations.
However, we show below that such arguments can also be made to
go through for simplicial complexes (with their infinite signature
described in Section~\ref{PrelimsSC}) and other similarly ``local'' classes
for infinite languages.
In this section we give a suitable definition of what it means to be a
local class for this purpose.

For context, we begin by recalling Oberschelp's notion of a parametric class.
\begin{defn}\label{Parametric}
Suppose $\Sigma$ is a relational signature.  A first-order sentence
$\phi$ over $\Sigma$ is \emph{parametric} if it is a conjunction of sentences
of the form
\[
\forall x_1\cdots\forall x_m((\bigwedge_{1\leq i<j\leq m}x_i\neq x_j)
\implies\psi)
\]
where $m>0$ and $\psi$ is a Boolean combination of terms
$R(y_1,\ldots,y_n)$ with $R\in\Sigma$ and 
$\{y_1,\ldots,y_n\}=\{x_1,\ldots,x_m\}$.
A class $\K$ of structures is said to be parametric if 
$\K=\Mod(\phi)$ for a parametric sentence $\phi$.
\end{defn}

Now to our variant.

\begin{defn}
Suppose $\Sigma$ is a relational signature.
A first-order sentence $\phi$ over $\Sigma$ is \emph{local} 
if it is of the form
\[
\forall x_1\cdots\forall x_m(R(x_1,\ldots,x_m)\implies\psi),
\]
where $m>0$ and $\psi$ is a quantifier-free formula.
\end{defn}
Note in particular that since $\phi$ is a sentence, the free variables of
$\psi$ are among $\{x_1,\ldots,x_m\}$; however we make no assumption that
they all appear in $\psi$.

\begin{defn}\label{Local}
Suppose $\Sigma$ is a relational language 
 with finitely many relations of each arity.  
A class $\K$ of $\Sigma$-structures is \emph{local} if there is a set
$\Phi\supset\Phi^\Sigma_{\textrm{GI}}$ of local sentences such that
$\K=\Mod\Phi$.
\end{defn}
Note that all of the axioms arising from our simplicial
complex axiom schemata are local, 
so the class of simplicial complexes is local.

An immediate connection between these definitions is the following.
The
subformula $\psi$ from Definition~\ref{Parametric} can frequently be 
written in the form $R(x_1,\ldots,x_m)\implies\psi')$, 
as in Definition~\ref{Local}, in which case 
the antecedent $\bigwedge x_i\neq x_j$ appearing in 
Definition~\ref{Parametric} is unnecessary if
the structures satisfy General Irreflexivity.
Thus, many natural examples of parametric classes are local.

On the other hand, there are parametric classes that are not local.
For example, as noted in \cite[Example~4.2.2~(c)]{EbF:FMT},
\[
\forall\text{ distinct }x_1,\ldots,x_k(R(x_1,\ldots,x_k)\land\lnot
R(x_1,\ldots,x_k))
\]
for $k$-ary $R$ is a parametric sentence with no models
of size $k$ or more.  
In contrast, every local class $\K$ contains structures of every finite size:
because of the form of local sentences, any finite set endowed with the
trivial (empty) interpretation for every relation will be a member of $\K$.

\begin{prop}
Any local class is a \Fraisse class.
\end{prop}
\begin{proof}
Suppose $\K$ is a local class.
The fact that HP holds is immediate from the fact that local sentences are
only universally quantified.
Given $B$ and $C$ in $\K$, $D$ as in JEP may be constructed as the 
disjoint union of $B$ and $C$, with the induced relations for tuples entirely
from $|B|$ or from $|C|$, and the relations always failing for mixed tuples.
Similarly, given $f_B:A\to B$ and $f_C:A\to C$ as in AP,
we may take $D$ with underlying set $|B|\dot\cup|C|/f_B(a)\sim f_C(a)$,
and relations induced from $B$ or $C$ if a tuple is entirely contained in one
(or both) of them, and not holding otherwise.
\end{proof}

\begin{lemma}\label{FramesinK}
For any local class $\K$, $A\in\K$, and $n\in\N$, $\fr{n}{A}\in\K$.
\end{lemma}
\begin{proof}
This is immediate from the definition of local sentences.
\end{proof}

Central to our probabilistic approach will be the following property,
which unites local and parametric classes.
\begin{description}
\item[Adoptive Property, AdP.]
Suppose $B\in\mathbf{K}$,
$A\in\K$ is a substructure of $B$ of cardinality $n$, and
$A'\in\K$ is a structure on $|A|$ with the same $(n-1)$-frame
$\f{A}{n-1}$ as $A$.  Then there is a structure $B'\in\K$ with underling set 
$|B|$
such that 
$\f{(B')}{n-1}=\f{B}{n-1}$, and for every $n$-ary relation $R\in\calL$,
\[
R^{B'} = (R^B\smallsetminus|A|^n)\cup R^{A'}.
\]
That is, $B'\sat R(b_1,\ldots,b_n)$ if and only if either 
\begin{enumerate}
\item one of the
$b_i$'s is not in $A$ and $B\sat R(b_1,\ldots,b_n)$,
or
\item $\{b_1,\ldots,b_n\}=|A|$ and $A'\sat R(b_1,\dots,b_n)$.
\end{enumerate}
\end{description}

The Adoptive Property is closely related to Koponen's 
notion of \emph{admitting substitutions}~\cite{Kop:APEPRlCS}, with the
difference being that the Adoptive Property takes a more frame-by-frame
approach.  As such, the Adoptive Property is more constrained in 
its antecedent for lower
dimensions, requiring that $\fr{n-1}{A}=\fr{n-1}{A'}$, and more liberal
in its consequent
in higher dimensions: AdP only requires that \emph{some} $B'$ with the
desired $n$-frame exists, with no demands placed on higher-arity relations,
whereas Koponen's substitutions define $B'$ in all dimensions from 
$B$ and $A'$.
For parametric classes
we may encapsulate this greater generality
as follows.

\begin{prop}
Every parametric class $\K$ enjoys the Adoptive Property.
\end{prop}
\begin{proof}
Because of the distinctness requirement on the variables 
in parametric sentences, we may take  
$R^{B'} = (R^B\smallsetminus|A|^k)\cup R^{A'}$
for every $k$-ary relation $R$ in the signature and every $k\in\N$
and clearly get another member of $\K$.
\end{proof}

The extra flexibility in higher dimensions for the Adoptive Property is
crucial for the example of simplicial complexes: if one removes a face from
a simplicial complex, one must also remove every higher dimensional face 
which contains it.

\begin{prop}
Every local class $\K$ enjoys the Adoptive Property.
\end{prop}
\begin{proof}
Let $A,A',B\in\K$ and $n\in\N$ be as in the statement of the Adoptive Property.
By Lemma~\ref{FramesinK} we may assume without loss of generality that
$B=\fr{n}{B}$.  Then $B'=\fr{n}{B'}$ constructed from $B$ as per the 
Adoptive Property is also a member of $\K$.
Indeed, consider any local sentence 
$\phi\equiv\forall x_1\cdots\forall x_m(R(x_1,\ldots,x_m)\implies\psi)$ in 
$\Phi_\K$.
If $m>n$ then $\phi$ holds vacuously in $B'$, and if $m<n$ then
$\phi$ holds in $B'$ because it holds in $B$.  If $m=n$, then if
$\{b_1,\ldots,b_n\}=|A|$, 
\begin{align*}
B'&\sat R(b_1,\ldots,b_n)\implies\psi(b_1,\ldots,b_n)\\
\text{ because}\quad
A'&\sat R(b_1,\ldots,b_n)\implies\psi(b_1,\ldots,b_n), 
\end{align*}
and if 
$\{b_1,\ldots,b_n\}\neq|A|$, 
\begin{align*}
B'&\sat R(b_1,\ldots,b_n)\implies\psi(b_1,\ldots,b_n)\\
\text{ because}\quad
B&\sat R(b_1,\ldots,b_n)\implies\psi(b_1,\ldots,b_n). 
\end{align*}
\end{proof}

The 
restriction of the cardinality of $A$ to $n$ does not make the
Adoptive Property weaker than it would otherwise have been.

\begin{lemma}\label{strAdP}
The following are equivalent for a class $\K$ of structures
over a locally finite relational signature.
\begin{enumerate}
\item\label{strAdPAdP} The Adoptive Property.
\item\label{strAdPstrAdP}
For every $B\in\mathbf{K}$,
$A\in\K$ a finite substructure of $B$ of cardinality at least $n$, 
and
$A'\in\K$ a structure on $|A|$ with the same $(n-1)$-frame
$\f{A}{n-1}$ as $A$,  
there is a structure $B'\in\K$ with underling set 
$|B|$
such that $\f{(B')}{n-1}=\f{B}{n-1}$ and
for every $n$-ary relation $R\in\calL$,
$R^{B'}=(R^B\smallsetminus|A|^n)\cup R^{A'}$.
\end{enumerate}
\end{lemma}
\begin{proof}
\eqref{strAdPstrAdP} $\metaimplies$ \eqref{strAdPAdP} 
is immediate.
For the converse, let $B$, $A$ and $A'$ be as in \eqref{strAdPstrAdP},
and let $B_0=B$.
Let $X_1,\ldots,X_{\|A\|\choose n}$ be an enumeration of the subsets of
$|A|$ of cardinality $n$.
Define $B_i$ recursively in $i$, letting $B_i$ equal the structure 
$B'$ obtained from the Adoptive Property applied to $B_{i-1}$ 
and $A'\restr X_i$.
Then $B_{\|A\|\choose n}$ satisfies the requirements for the structure 
$B'$ as in \eqref{strAdPstrAdP}.
\end{proof}

\subsection{Examples of local classes}\label{Egs}

The class of 
simplicial complexes is not the only natural example of a local class.

\subsubsection*{Hypergraphs}

The definition of a \emph{hypergraph} 
varies from reference to reference;
we shall take a hypergraph $H$ on a set $X$ to be simply
a subset of $\fin{X}=\{Y\subset X\st Y\text{ is finite}\}$.
As for simplicial complexes, we refer to $Y\in H$ as a \emph{face} of $H$.

In our model-theoretic setting, this definition is encapsulated as follows.
\begin{defn}
A \emph{hypergraph} is formally a $\Ssc$-structure satisfying the
General Irreflexivity and Symmetry schemata of Definitions~\ref{GIdefn} and
\ref{SCdefn} above.
\end{defn}
Thus, the difference from the definition of a simplicial complex is that we
do not require that the Subset Closure axiom schema be satisfied by
hypergraphs.
Clearly the class of finite hypergraphs is a local class.

Note that if all of the faces of a hypergraph $H$ 
are of the same cardinality $d$, it may be identified with the simplicial 
complex with complete $(d-1)$-frame.
Indeed, Rado's construction \cite{Rad:UGUF}
of the \Fraisse limit of $d$-dimensional
simplicial complexes with complete $(d-1)$-skeleton, which was mentioned in
the introduction, is actually cast in these terms.

\subsubsection*{Sperner families}

Recall that a \emph{Sperner family $A$} on a set $X$ is an antichain in $\fin{X}$, that is a subset 
$A$ of $\fin{X}$ with the property that if $Y,Z \in A$ 
and $Y \subset Z$, 
then $Y=Z$.  This too may be formalised to yield an example of a local class.
\begin{defn}
A \emph{Sperner family} is a hypergraph satisfying  
the following further axiom schema.
\begin{description}
\item[Non-subset:]
For every $m < n$,
\[
\forall x_0\cdots\forall x_n(S_n(x_0,\ldots,x_n)\implies
\lnot S_m(x_0,\ldots,x_m)).
\]
\end{description}
\end{defn}
Again, it is clear that Sperner families form a local class.
It should be noted, on the other hand, 
that simplicial complexes are more closely related to 
Sperner families than just via our formalism: a simplicial complex may be 
identified with the Sperner family of its minimal non-faces (or for finite
simplicial complexes, the Sperner family of its maximal faces).

\section{First order theory}
%
As mentioned above, with our definition of local classes in place,
a range of model-theoretic consequences may be obtained by 
suitably modifying standard arguments.  

Let us fix $\Sigma=\{R_{m,k}\st 0<m\in\N,1\leq k\leq k_m\}$,
a locally finite relational signature,
where each $R_{m,k}$ is an $m$-ary relation symbol, 
so that $k_m$ is the number of $m$-ary relations in $\Sigma$.
Let $\K=\Mod(\Phi_\K)$ be a
local class of finite $\Sigma$-structures, and 
let 
$\R$ be the \Fraisse limit of $\mathbf{K}$.
Of course, for concreteness the reader may think of our motivating example,
with $\Sigma=\Ssc$ and $\K$ the class of finite simplicial complexes.
Let $\Thr$ denote the complete first order theory of $\R$ over
the signature $\Sigma$, that is,
the set of all sentences over $\Sigma$ that are true in $\R$.
A variety of nice results are known regarding the theory of the 
\Fraisse limit of a
\Fraisse class for a finite language; 
we show here that some of the most important results also
hold when $\K$ is a \genskelish \Fraisse class.

\begin{thm}\label{RecCatQE}
(i) Every countable model of $\Thrsc$ is isomorphic to $\R$
(that is, $\Thrsc$ is \emph{countably categorical}).\\
(ii) $\Thrsc$
satisfies
\emph{quantifier elimination:}
for every formula $\phi$ of $\Lsc$ with free variables $x_1,\ldots,x_n$,
there is a formula $\psi$ involving no quantifiers such that
\[
\forall x_1\cdots\forall x_n \;\; \Bigl( \phi \longleftrightarrow \psi \Bigr)
\]
is in $\Thrsc$.
\end{thm}
\begin{proof}
The proofs of these facts are much like 
the standard proofs for the finite $\Sigma$ case, 
as for example in~\cite[Theorem~7.4.1]{Hod:MT},
but using the fact that $\K$ is \genskelish in place of finiteness.

\noindent(i) We start by axiomatising $\Thr$.
This will be by means of a formalisation of the 
\homogeneity of $\R$, as in Definition~\ref{wkhom}.  
By induction and using the Hereditary Property, 
it is sufficient to consider structures $A$ and $B$ in $\K$ such that
$B$ is a one point extension of $A$.
Recall that we assume that every member of $\K$ has underlying set a subset of
$\N$, 
and that $\K$ is assumed to be closed under isomorphism of such structures.
Thus, we may assume without loss of generality that
$|B|=\{1,\ldots,n\}$ and $A$ is the restriction of $B$ to $\{1,\ldots,n-1\}$.
For every $m\leq n$, every $m$-ary relation symbol $R_{m,k}\in\Sigma$,
and all $i_1,\ldots,i_m$ between $1$ and $n$ inclusive, let
$\bar{R}_{m,k}^{B:i_1,\ldots,i_m}$ denote the relation symbol
$R_{m,k}$ if $R_{m,k}(i_1,\ldots,i_m)$ holds in $B$, and $\lnot R_{m,k}$ if not.
Now let $\phi_B$ denote the \emph{extension axiom} corresponding to
$A\hookrightarrow B$, that is, the formal sentence that can be written as
\begin{multline*}
\forall x_1\cdots\forall x_{n-1}\\
\Bigg(
\Big(
\bigwedge_{1\leq i\neq j\leq n-1}\!\!\!\!\!x_i\neq x_j\ \land\ 
\bigwedge_{m=1}^{n-1}\bigwedge_{k=1}^{k_m}\bigwedge_{1\leq i_1,\ldots,i_m<n}
\bar{R}^{B:i_1,\ldots,i_m}_{m,k}(x_{i_1},\ldots,x_{i_m})
\Big)\\
\implies
\exists x_n\Big(
\bigwedge_{i=1}^{n-1}x_i\neq x_n\ \land\ 
\bigwedge_{m=1}^{n}\bigwedge_{k=1}^{k_m}\bigwedge_{1\leq i_1,\ldots,i_{m}\leq n}
\bar{R}^{B:i_1,\ldots,i_{m}}_{m,k}(x_{i_1},\ldots,x_{i_{m}})
\Big)
\Bigg).
\end{multline*}
Thus, $\phi_B$ formally expresses the statement that if there is a
substructure isomorphic to $A$, then there is a substructure extending it by
one vertex which is isomorphic to $B$
(we of course take empty conjunctions to be true, so that 
in the $n=1$ case $\phi_B$
reduces to a simple statement expressing the existence of a substructure
isomorphic to $B$).
These formulas $\phi_B$ 
may be thought of as generalisations of the well-known property of the 
infinite random graph, that given any two finite sets of vertices, there is a 
vertex adjacent to every vertex in the first set and not adjacent to any vertex
in the second set.  In the random graph case the structure of $A$ 
(that is, the induced subgraph on the union of the two sets) is irrelevant,
but in our setting this will not in general be the case.

Note that in stating $\phi_B$ we have made crucial use of the fact that $\K$ 
is \genskelish. 
In particular, the conjunction over $k$ is finite for each $m$
because $\Sigma$ is locally finite, and the fact that we are free to ignore
higher-arity relations follows from General Irreflexivity.

The sentences $\phi_B$ encapsulate the \homogeneity of $\R$;
moreover by induction they show that $\K$ is a subset of the age of any
structure that satisfies all of them.
Thus, any structure satisfying 
\[
\Phi_\R = \Phi_\K\,\cup
\left\{\phi_B \; \Bigl| \; 
B\in\K \; \land \; \exists n\in\N \;\; \Bigl( |B|=\{1,\ldots,n\} \Bigr) \right\}
\]
has age equal to $\K$.
Clearly $\R\sat\Phi_\K$: 
any (local) sentence $\phi\in\Phi_\K$ has negation of the form
\(
\exists x_1\cdots\exists x_m(\psi)
\)
where $\psi$ is quantifier-free,
so a witnessing tuple in $\R$ would also witness the failure of $\phi$ in
a finite substructure.
Thus, $\Phi_\R$ is a subset of $\Thr$, such that that any model of $\Phi_\R$
is a weakly homogeneous structure with age $\K$, and therefore
is isomorphic to $F$. 
So $\Thr$ is countably categorical, and by the G\"odel completeness theorem
we have that $\Phi_\R$ provides a complete axiomatisation of $\Thr$.

\noindent(ii) Suppose $\phi$ is a formula over $\Sigma$
with free variables $x_1,\ldots,x_n$.  We shall exhibit a formula
$\psi$ with free variables $x_1,\ldots,x_n$ and no quantifiers, such that 
relative to $\Thrsc$, $\psi$ is equivalent to $\phi$.
If there is no tuple $\mathbf{r}=(r_1,\ldots,r_n)$ from 
$\R$ such that $\phi(\mathbf{r})$ holds in $F$, then 
any tautologically false 
formula such as $\lnot(x_1=x_1)$ will do for $\psi$.  
Otherwise,
consider any tuple $\mathbf{r}=(r_1,\ldots,r_n)$
from $\R$ such that $\phi(\mathbf{r})$ holds in $\R$.
A standard property of \Fraisse limits is that they are
\emph{ultrahomogeneous}: any isomorphism of finite substructures of $\R$
extends to an automorphism of $\R$ 
(see for example Hodges~\cite[Lemma~7.1.4]{Hod:MT}).
Thus,
for any 
$\mathbf{r'}$ with a coordinate-respecting isomorphism $f$ from
$\R\restricted\mathbf{r}$ to
$\R\restricted\mathbf{r'}$,
we have an automorphism $\bar f$ of $\R$ 
extending $f$.
Since automorphisms preserve all formulas, $\phi(\mathbf{r'})$ will 
therefore also hold in $\R$.
Hence, whether $\phi$ holds on a given $\mathbf{r}$ depends entirely on the
induced substructure on the elements of $\R$ appearing in
$\mathbf{r}$.  
So 
consider 
the set
\[
S_\phi=
\left\{
\begin{tabular}{r|@{\;\;\;}l}
$(A,\mathbf{a})$
&
\parbox{250pt}{$A$ is a $\Sigma$-structure on $\{1,\ldots,n\}$ in $\K$ 

$\land$ $\mathbf{a}=(a_1,\ldots,a_n)\in\{1,\ldots,n\}^n$ 

$\land$ there is an embedding $f \colon A\hookrightarrow \R$
such that $\phi(f(a_1),\ldots,f(a_n))$}
\end{tabular}
\right\}
\]
(there are of course redundancies in $S_\phi$ that could be eliminated,
but it does not seem worth the notational hassle).
Then $S_\phi$ is finite since $\K$ is \genskelish, and
we may take our $\psi$ to be the formula
\[
\bigvee_{(A,\mathbf{a})\in S_\phi}
\bigwedge_{m=1}^{n}
\bigwedge_{k=1}^{k_m}
\bigwedge_{1\leq i_1<\cdots<i_m\leq n}
\bar{R}^{A:a_{i_1},\ldots,a_{i_m}}_{m,k}(x_{i_1},\ldots,x_{i_m})
\]
(with 
\(\bar{R}^{A:a_{i_1},\ldots,a_{i_m}}_{m,k}\) as in part (i)).
This $\psi$ is a quantifier-free equivalent of $\phi$ relative to 
$\Thrsc$.
\end{proof}


\section{A probabilistic approach}

In this section we shall consider a probabilistic characterisation
of the \Fraisse limit of a \genskelish class $\K$.
This will give rise to our \zo for first order
sentences about 
elements of $\K$, relative to the appropriate measure on the set
of members of $\K$ whose underlying set has $m$ elements.  
We shall start by discussing the case of simplicial complexes in
order to convey the main ideas more concretely, and then move on to
the general case of a \genskelish \Fraisse class~$\K$.

Probably the best-known description of the random graph is as the
graph almost surely obtained (up to isomorphism) by including each possible
edge with probability one half.  Our generalisation of this to the
simplicial complex context is as follows.  
Having constructed by tossing a fair coin a graph
(in the infinite case, almost surely the random graph) as the 1-skeleton
of a simplicial complex, we continue to make all decisions in the 
construction of the complex through all higher
dimensions by tossing a fair coin.  That is, given a triple
$v_0, v_1, v_2$ of vertices such that 
$\{v_0,v_1\}$, 
$\{v_0,v_2\}$ and 
$\{v_1,v_2\}$ are all edges in the complex, we toss a fair coin
to determine whether $\{v_0,v_1,v_2\}$ is a 2-face in the simplicial
complex; and so on through all dimensions.
Of course, to see whether a decision needs to be made for a given
$n$-tuple of vertices, one only needs to know the decisions for 
subsets
of those vertices.  
Despite the fact that the resulting probability measure for 
simplicial complexes seems very natural, 
it appears not to have been considered in detail before,
although the version for finite simplicial complexes
has been proposed~\cite{MSS:SMSA} for 
use in social aggregation modelling,
and for $d$-dimensional simplicial complexes for fixed $d$ 
this can be seen as a case of Koponen's 
\emph{uniformly $(\mathbf{C}_0,\ldots,\mathbf{C}_{d+1})$-conditional
probability measure}~\cite[Definition~6.2]{Kop:APEPRlCS}.

We now work towards a formalisation of this for local classes in general.
Let $\K=\Mod(\Phi_\K)$ be a \genskelish class for a locally finite relational
signature $\Sigma$. 
We continue in our assumption that
each member $A$ of $\K$ has underlying set $|A|\subset\N$
and $\K$ is closed under isomorphism of such structures; 
recall also that
for any structure $S$ we denote by $\f{S}{k}$ the $k$-frame of $S$.
\begin{defn}
For $i\in\N$, we denote by 
$\K^\infty$ the set of $\Sigma$-structures $M$
with underlying set $\N$ such that $M\sat\Phi_\K$.
Similarly, we denote by
$\K^i$ the set of members of $\K$ with underlying set $\{0,\ldots,i-1\}$, 
we let $\K^{\subset i}$ denote the set of members of $\K$ with underlying set
a subset of $\{0,\ldots,i-1\}$, 
and for notational convenience we set $\K^{\subset\infty}=\K$.
\end{defn}
The last part of this definition is apropos because
we have the following equivalent characterisation of $\K^\infty$.
\begin{lemma}
Let $M$ be a $\Sigma$-structure with $|M|=\N$. Then
$M$ is in $\K^\infty$ if and only if
every finite substructure of $M$ is in $\K$.
\end{lemma}
\begin{proof}
This is immediate from the fact that local sentences are only universally 
quantified.
\end{proof}

For $i \in \N \cup \{\infty\}$ 
we define a topology on $\K^i$ having
as a base of open sets the collection of sets of the form
\[
O_i(A,k)=\{S\in\K^i\st\f{(S\restr|A|)}{k}=\f{A}{k}\}
\]
for $A\in\K^{\subset i}$ and $k\in\N$.

If $i$ is a natural number, then the topology on $\K^i$ defined by 
this base is discrete by Lemma~\ref{Onnnframe}.  Moreover, there 
is a natural map $\K^i \to \K^{i-1}$ restricting a structure on 
$\{0 , \ldots , i-1 \}$ to the induced substructure on 
$\{0 , \ldots , i-2\}$.  The space $\K^\infty$ is the inverse 
limit along these maps 
of the spaces $\K^i$ as $i$ varies, and thus is profinite.

Note that 
$\K^\infty$ with the above
topology is compact: 
indeed, it can be viewed as a closed subset of the
compact (using Lemma~\ref{fincardn}) Hausdorff space
$\prod_{i\in\N}\K^i$.
Furthermore, every open cover of $\K^\infty$ 
has a finite subcover with a finite \emph{disjoint} refinement. 
This allows us to simply check finite additivity in order to apply 
Carath\'eodory's Extension Theorem.

Let $N(S,k)$ denote the number 
of distinct $(k+1)$-frames 
of structures $T\in\K$ on the underlying set $|S|$ of $S$
such that $\f{T}{k}=\f{S}{k}$.

\begin{defn}\label{muBasis}
For all $i\in\N\cup\{\infty\}$, define $\mu_i$ on the base sets
$O_i(A,k)$ $(A\in\K^{\subset i})$ 
recursively in $k$ by setting $\mu_i(O_i(A,0))=1$ and
\[
\mu_i(O_i(A,k+1))
=\frac{1}{N(A,k)}\,\mu_i(O_i(A,k)).
\]
\end{defn}
Observe that if $i\leq j$ and $A\in\K^{\subset i}\subset\K^{\subset j}$,
then $\mu_i(O_i(A,k))=\mu_j(O_j(A,k))$, and that 
for all $A\in\Ks{i}$ and $k\in\N$, $\mu_i(O_i(A,k))\neq0$.

There are two parameters to our base sets: the substructure $A$ and the
frame level $k$.  Definition~\ref{muBasis} 
is such that $\mu_i$ is clearly additive as
one changes $k$, but we need to check that it is also additive for varying $A$.
We introduce some notation to help with the proof of this fact.
\begin{defn}
For $X$ a finite subset of $\N$, 
$A$ an $\calL$-structure with underlying set 
$|A|\subset X$, and $k\leq l$ natural numbers such that
$k\leq\|A\|$ and $l\leq\|X\|$, define
\[
\K^{X,l}_{A,k}=\{\f{B}{l}\st B\in\K\land|B|=X\land
\f{B}{k}\restr|A|=\f{A}{k}\}.
\]
\end{defn}
For example, $N(S,k)=\|\K^{|S|,k+1}_{S,k}\|$.
\begin{lemma}\label{uniformity}
Suppose $S\in\K$, $A$ is a substructure of $S$, and 
$1\leq k\leq\|A\|$.
Then for all $A'\in\K^{|A|,k}_{A,k-1}$,
\[
\left\|\K^{|S|,k}_{S,k-1}\cap\K^{|S|,k}_{A',k}\right\|=
\left\|\K^{|S|,k}_{S,k-1}\cap\K^{|S|,k}_{A,k}\right\|.
\]
In particular, 
\[
\left\|\K^{|S|,k}_{S,k-1}\right\|=
\left\|\K^{|A|,k}_{A,k-1}\right\|
\left\|\K^{|S|,k}_{S,k-1}\cap\K^{|S|,k}_{A,k}\right\|=
N(A,k-1)
\left\|\K^{|S|,k}_{S,k-1}\cap\K^{|S|,k}_{A,k}\right\|.
\]
\end{lemma}
\begin{proof}
Lemma~\ref{strAdP} provides a natural bijection between
$\K^{|S|,k}_{S,k-1}\cap\K^{|S|,k}_{A',k}$ and
$\K^{|S|,k}_{S,k-1}\cap\K^{|S|,k}_{A,k}$.
\end{proof}

\begin{lemma}
Suppose $A\in\K^{\subset i}$, 
$0\leq k\leq\|A\|$, and 
$X$ is a finite subset of $\{j\in\N\st j<i\}$ such that $|A|\subset X$.
Then
\[
\mu_i(O_i(A,k))=\sum_{S\in\K^{X,k}_{A,k}}\mu_i(O_i(S,k))
\]
\end{lemma}
\begin{proof}
The proof is by induction on $k$, using the Adoptive Property
in the form of Lemma~\ref{uniformity}.
For $k=0$ the statement holds because there is a unique $0$-frame on
any given underlying set: the structure with the empty interpretation for
every relation.

So suppose the statement is true for values up to and including $k-1$.
Hence,
\begin{align*}
\mu_i(O_i(A,k-1))&=\sum_{S\in\K^{X,k-1}_{A,k-1}}\mu_i(O_i(S,k-1))\\
N(A,k-1)\,\mu_i(O_i(A,k))&=
\sum_{S\in\K^{X,k-1}_{A,k-1}}N(S,k-1)\,\mu_i(O_i(\bar S,k)),
\end{align*}
where for every $S\in\K^{X,k-1}_{A,k-1}$, $\bar S$ is a $k$-frame of the
form $\fr{k}{B}$ for some $B\in\K$ with $\fr{k-1}{B}=S$;
by Definition~\ref{muBasis}, the expression is independent of the choice of
$\bar S$.
We have
\begin{align*}
N(A,k-1)\,\mu_i(O_i(A,k))&=
\sum_{S\in\K^{X,k-1}_{A,k-1}}\left\|\K^{X,k}_{S,k-1}\right\|
\mu_i(O_i(\bar S,k))\\
\mu_i(O_i(A,k))&=
\sum_{S\in\K^{X,k-1}_{A,k-1}}
\left\|\K^{X,k}_{S,k-1}\cap\K^{X,k}_{A,k}\right\|
\mu_i(O_i(\bar S,k))\\
&=\sum_{\bar{S}\in\K^{X,k}_{A,k}}\mu(O_i(\bar{S},k)).
\end{align*}
\end{proof}

This result easily yields finite additivity for $\mu_i$ on the
ring of sets generated by the 
base sets $O_i(A,k)$.  
Since all of the sets in this ring are clopen, this is equivalent to
$\sigma$-additivity, and so applying the Carath\'eodory Extension Theorem
in the $i=\infty$ case
we may make the following definition.

\begin{defn}
For all $i\in\N\cup\{\infty\}$, the \emph{frame-wise uniform measure}
$\mu_i$ on $\K^i$ is the probability measure induced by 
Definition~\ref{muBasis}.
\end{defn}

The Adoptive Property tells us that whether a certain $n$-frame can occur
on a given subset $X$ of a structure depends only on the $(n-1)$-frame on
$X$, and is independent of ``what happens elsewhere''.  
With the frame-wise uniform measure, we also have this independence in the
probability theory sense of the word.  
The following Lemma is indicative of this.

\begin{lemma}\label{indep}
For and $B\in\K$ and any $k<\card{B}$,
\[
N(B,k)=\prod_{X\subset|B|,\|X\|=k+1}N(B\restr X,k).
\]
\end{lemma}
\begin{proof}
The $(k+1)$-frame of a structure is of course determined by the $(k+1)$-frame
of each of its $(k+1)$-element substructures.
By the Hereditary Property, the right hand side of the equation is therefore
the maximum possible value for $N(B,k)$, and the Adoptive Property
shows that indeed every possibility 
occurs in $\K$.
\end{proof}

\begin{lemma}\label{usefulindep}
Suppose $A, B\in\K$, $k\leq\|A\|$, $l\leq\|B\|$, let
$Y=|A|\cap|B|$ and \(m=\min(l,\|Y\|),\)
and suppose
$k\geq m$ and 
$\fr{m}{(A\restr Y)}
=\fr{m}{\mbox{$(B\restr Y)$}}$.
Then
\[
\frac{\mu_\infty(O_\infty(A,k))}{\mu_\infty(
O_\infty(A\restr Y,m))}=
\frac{\mu_\infty(O_\infty(A,k)\cap O_\infty(B,l))}{\mu_\infty(O_\infty(B,l))}.
\]
More generally,
suppose further that $C_i\in\K$ and $n_i\in\N$ for $i$ in a finite set $I$
are such that $n_i\leq\|C_i\|$,
$|A|\cap|C_i|=\emptyset$ for all $i\in I$,
\[
\fr{\min(l,n_i,\||B|\cap|C_i|\|)}{(B\restr|B|\cap|C_i|)}
=\fr{\min(l,n_i,\||B|\cap|C_i|\|)}{(C_i\restr|B|\cap|C_i|)}
\]
for all $i\in I$, and
\[
\fr{\min(l,n_i,\||C_i|\cap|C_j|\|)}{(C_i\restr|C_i|\cap|C_j|)}
=\fr{\min(l,n_i,\||C_i|\cap|C_j|\|)}{(C_j\restr|C_i|\cap|C_j|)}
\]
for all $i\in I$.
Let $N$ be sufficiently large that $A, B$ and all $C_i$ are in $\K^{\subset N}$.
Then 
\[
\frac{\mu_N(O_N(A,k))}{\mu_N(
O_N(A\restr Y,m))}=
\frac{\mu_N(O_N(A,k)\cap O_N(B,l)
\cap\bigcap_{i\in I} O_N(C_i,n_i))}{\mu_N(O_N(B,l)
\cap\bigcap_{i\in I} O_N(C_i,n_i))}.
\]
\end{lemma}
\begin{proof}
This is clear by induction on $k$, starting from $0$ if 
$|A|\smallsetminus Y$ is nonempty and $m$ otherwise.
\end{proof}

Erd\H{o}s and Renyi~\cite{ErR:AsG}
showed that almost every countably infinite graph is isomorphic to the
infinite random graph.  
With our frame-wise uniform measure $\mu_\infty$ we have
the analogous result for members of $\K^\infty$.

\begin{thm}\label{aaR}
Under the frame-wise uniform measure, 
almost every structure in $\K^\infty$
is isomorphic to the \Fraisse limit $\R$ of $\K$.  
That is,
\[
\mu_\infty(\{S\in\K^\infty\st S\cong\R\})=1.
\]
\end{thm}
\begin{proof}
The countable categoricity of $\Thr$ given by Theorem~\ref{RecCatQE}
tells us that 
a countable $\Sigma$-structure is isomorphic to $\R$
if and only if it satisfies the axioms $\Phi_\R$ for $\Thr$.
By definition the elements of $\K^\infty$
satisfy the axioms $\Phi_\K$, so
it suffices to show that almost every member of $\K^\infty$ satisfies
all of the extension axioms
$\phi_B$ as $B$ varies over members of $\K$.
There are only countably many such $B$, so since $\mu_\infty$ is 
countably additive, it suffices to show that
for each such $B$, 
\[
\mu_\infty(\{S\in\K^\infty\st S\nvDash\phi_B\})=0.
\]
So suppose $B$ is a member of $\K$ with underlying set $|B|=\{0,\ldots,n\}$, 
and $A$ is the substructure induced on $\{0,\ldots,n-1\}$.  
For each
$n$-tuple $\mathbf{j}=(j_0,\ldots,j_{n-1})\in\N^n$ with distinct elements,
let $\K^\infty_{A,\mathbf{j}}$ denote 
the set of members of $\K^\infty$ into which 
$A$ embeds by the map $i\mapsto j_i$, that is,
in the notation of Theorem~\ref{RecCatQE},
\[
\K^\infty_{A,\mathbf{j}}=\Big\{S\in\K^\infty\,\Big|\,
\bigwedge_{m=1}^{n-1}\bigwedge_{k=1}^{k_m}
\bigwedge_{0\leq i_1,\ldots,i_m\leq n-1}
\bar{R}_{m,k}^{A:i_1,\ldots,i_m}(j_{i_1},\ldots,j_{i_m})\Big\}.
\]
By $\sigma$-additivity again, it suffices to show that for every 
$\mathbf{j}=(j_0,\ldots,j_{n-1})\in\N^n$ with distinct elements,
\begin{multline*}
\mu_\infty\Bigg(\bigg\{S\in\K^\infty_{A,\mathbf{j}}\st 
\forall x_n
\Big(
\bigvee_{i=0}^{n-1}j_i=x_n\ \lor\\
\neg
\bigwedge_{d=0}^{n}\bigwedge_{i_0<\cdots<i_{d-1}<n}
\bar{R}^{B:i_0,\ldots,i_{d-1},n}_{d}(j_{i_0},\ldots,j_{i_{d-1}},x_n)
\Big)
\bigg\}\Bigg)=0,
\end{multline*}
that is, the probability that no one-point extension of the image of $A$ 
in the structure is isomorphic to $B$ is 0.
But now for members of $\K^\infty_{A,\mathbf{j}}$, the
probability that the substructure induced on $\mathbf{j}\cup\{j_n\}$ 
is isomorphic to $B$
for a given $j_n$ not in $\mathbf{j}$ is non-zero, and takes the same
value for all such $j_n$. 
Moreover, for a given $j_n$ it is
independent of whether it holds for
other elements of $\N$ not in $\mathbf{j}$, by Lemma~\ref{usefulindep}.
Hence, the above measure is indeed 0, and we are done.
\end{proof}

\subsection{\Zo}

In this section we show that, using the frame-wise uniform measure, there is a
\zo for arbitrary simplicial complexes, and likewise for the members of
any local class.
As mentioned above,
Blass and Harary~\cite{BlH:PAAGC} have shown that there is a \zo for
simplicial complexes of dimension at most some given bound $d$, 
using for each $n$ the usual uniform measure on the set of 
at most $d$-dimensional simplicial complexes on $n$ vertices
(that is, simply counting the simplicial complexes, with no
regard for their structure).  
The restriction on the dimension that they
impose is crucial for making sense of their {\zo} --- in their measure,
almost every simplicial complex has a complete $(d-1)$-skeleton.
The \zo we obtain below for the frame-wise uniform measure imposes
no such restriction, and indeed, in our \zo the probability that even the
underlying graph is complete converges to 0.  In fact, in the frame-wise 
uniform measure, the underlying graph converges to the countable random graph.

Interestingly, in obtaining their \zo Blass and Harary use 
axioms that generalise the extension axioms for the random graph, as
our axioms $\phi_B$ also do.  However, their generalisation is more
direct, involving a one-point extension in terms of two sets for which
simplices should or should not be added.  The simplicity of this generalisation
is what forces the $(d-1)$-skeleton to be complete, but of course it
turns out that this is the right thing to do for the uniform measure.
Our axioms $\phi_B$ reflect more of a \Fraisse limit perspective, 
with the result that the \zo that we obtain
is for the frame-wise uniform measure.
As such, our result is in a sense closer to the original \zos of
Glebskii, Kogan, Liogon'kii and Talanov~\cite{GKLT:RDR} and
Fagin~\cite{Fag:PFM} than it is to the \zo of Blass and Harary.

We continue to assume we have a fixed local class $\K$ over a locally finite
relational signature $\Sigma$ with \Fraisse limit $\R$,
and take $\mu_N$ to be the frame-wise uniform measure on $\K^N$ for 
$N\in\N\cup\{\infty\}$, as defined above.
\begin{thm}\label{01Law}
Let $\phi$ be a first-order sentence over $\Sigma$.
As $N$ goes to infinity, 
$\mu_{N}(\{S\in\K^N\st S\sat\phi\})$ converges,
with limiting value $1$ if $\phi$ holds of $\R$, and $0$ if not.
\end{thm}
\begin{proof}
Since $\Thr$ is complete and axiomatised by $\Phi_\R$ as in the proof of 
Theorem~\ref{RecCatQE}, 
every sentence $\psi$ in the language $\calL$ will be provable or disprovable
from $\Phi_\R$.  Of course any proof (of either $\psi$ or $\lnot\psi$)
will only involve finitely many axioms from $\Phi_\R$.
Therefore, if each axiom in $\Phi_\R$ holds in members of $\K^N$ with probability approaching
1, then every sentence $\psi$ true in $\R$
will hold with probability approaching 1, and the negations of such statements
will hold with probability approaching 0.
We thus restrict attention to sentences in $\Phi_\R$.

The idea of the proof is as for Theorem~\ref{aaR},
in that as $N$ goes to infinity the probability that a given embedding of
some $A$ does not extend to an embedding of $B$ drops to $0$.  However,
in this finite case we must work harder to account for potential new embeddings
of $A$.  We use a proof reminiscent of Fagin~\cite[Theorem~2]{Fag:PFM}.

Suppose $A\in\K$ has cardinality $n-1$ 
and $B\in\K$ is a one-point extension of $A$, and as 
in the proof of Theorem~\ref{RecCatQE},
let $\phi_B$ denote the corresponding extension axiom.
Let $\psi_B(x_1,\ldots,x_{n-1})$ denote
\begin{multline*}
\bigwedge_{m=1}^{n-1} \;\; \bigwedge_{k=1}^{k_m} \;\; \bigwedge_{1\leq i_1,\ldots,i_m<n}
\bar{R}^{B:i_1,\ldots,i_m}_{m,k}(x_{i_1},\ldots,x_{i_m})
\\
\land
\forall x_n\left(
\bigvee_{i=1}^{n-1}x_i= x_n\ \lor\ 
\bigvee_{m=1}^{n} \;\; \bigvee_{k=1}^{k_m} \;\; \bigvee_{1\leq i_1,\ldots,i_{m}\leq n}
\lnot\bar{R}^{B:i_1,\ldots,i_{m}}_{m,k}(x_{i_1},\ldots,x_{i_{m}})
\right),
\end{multline*}
that is, the negation of the $\phi_B$ without the quantification or
distinctness requirement on the variables $x_1,\ldots,x_{n-1}$.
Then for any $N$ we have
\begin{align*}
\mu_N&(\{S\in\K^N\st S\nsat\phi_B\})\\
&=
\mu_N\Big(\Big\{S\in\K^N\st S\sat
\exists x_1\cdots\exists x_{n-1}
\Big(
\bigwedge_{1\leq i\neq j\leq n-1}\!\!\!\!\!x_i\neq x_j\ \land\ 
\psi_B(x_1,\ldots,x_{n-1})
\Big)\Big\}\Big)\\
&\leq
\sum\Big\{\mu_N(\{S\in\K^N\st S\sat\psi_B(a_1,\ldots,a_{n-1})\})\st
0\leq a_1,\ldots,a_{n-1}\leq N-1 \\
&\phantom{=
\sum\Big\{\mu_N(\{S\in\K^N\st S\sat\psi_B(a_1,\ldots,a_{n-1})\})\st}
\ \land
a_i\neq a_j \text{ for }i\neq j\Big\}\\
&=
{N\choose n-1}\mu_N(\{S\in\K^N\st S\sat\psi_B(1,\ldots,n-1)\})
\qquad\text{ by symmetry}\\
&\leq N^{n-1}\mu_N(\{S\in\K^N\st S\sat\psi_B(1,\ldots,n-1)\})\\
&=N^{n-1}\mu_N\Big(\Big\{S\in\K^N\st
\forall y(\bigwedge_{i=1}^{n-1}y\neq i \implies\\
&\qquad\qquad\qquad\qquad
\bigvee_{m=1}^{n} \;\; \bigvee_{k=1}^{k_m} \;\; \bigvee_{1\leq i_1,\ldots,i_{m-1}\leq n}
\!\!\!\!
\lnot\bar{R}^{B:i_1,\ldots,i_{m-1},n}_{m,k}(x_{i_1},\ldots,x_{i_{m-1}},y))
\Big\}\Big)
\end{align*}
Let $p=\mu_N(B,\|B\|)/\mu_N(A,\|A\|)$.  Then applying 
Lemma~\ref{usefulindep}, the last line above becomes
$N^{n-1}\mu_N(O_N(A,\|A\|))(1-p)^{N-n+1}$.
Since $p$ is non-zero, this converges to $0$ as $N$ goes to infinity,
so $\lim_{N\to\infty}\mu_N(\{S\in\K^N\st S\nsat\phi_B\})=0$, as required.
\end{proof}
We shall follow the convention that 
``almost all simplicial complexes satisfy $\psi$'' is a shorthand for
``as $N$ goes to infinity, the measure of the set of simplicial complexes
satisfying $\psi$ goes to 1''.

Let us consider the particular case of
simplicial complexes.  
The way we have set up our formalism,
given an underlying set $V$, only a subset of $V$ will be vertices
(that is, singletons, or 0-dimensional faces) in our simplicial complex.
In some contexts this is appropriate, but to fit better with conventions
in the area, it is preferable to have all members of $V$ as vertices in 
any simplicial complex ``on $V$''.  This is actually easy to achieve:
we simply omit the relation $S_0$ from our signature, and otherwise proceed
unchanged.
Unless otherwise stated, this will be our approach henceforth.

It is obvious that our \zo differs from that of Blass and 
Harary; for example, for the the frame-wise uniform measure, almost all
simplicial complexes have an incomplete underlying graph.
Indeed,
a big advantage of the frame-wise uniform measure is that it respects
the probabilities of underlying graphs, with the result that we can
draw on the large body of knowledge regarding the random graph.  For
example, the result of Erd\H{o}s and Renyi~\cite{ErR:AsG} that almost
every graph is rigid lifts to give us the following.
\begin{prop}
Under the frame-wise uniform measure, 
almost all simplicial complexes have trivial automorphism group.
\end{prop}
\begin{proof}
Any automorphism of a simplicial complex is also an automorphism on the
underlying graph.  But now
\[
\mu_N(\{S\in\Ksc^N\st S^1\text{ has trivial automorphism group}\})
\]
equals (by definition)
the proportion of graphs on $N$ vertices with trivial automorphism group,
so 
\begin{multline*}
\mu_N(\{S\in\Ksc^N\st S\text{ has trivial automorphism group}\})\\
\geq
\mu_N(\{S\in\Ksc^N\st S^1\text{ has trivial automorphism group}\})
\end{multline*}
which converges to 1 as $N$ goes to infinity.
\end{proof}
We note that Bollob\'as and Palmer~\cite{BoP:AASCP} obtained the same result for
the usual uniform measure on $\Ksc^N$. 
Observe that this result also gives us a \zo for an unlabelled
form of the frame-wise uniform measure, 
with the infinite random simplicial complex $\RSC$ 
still acting as the 
oracle for truth.

\section{Algebraic properties}
\label{algebra}

As mentioned in the introduction, the automorphism groups of 
homogeneous structures is an important area of current research.
In this section we present some basic results about 
the automorphism groups of $\RSC$ and similar \Fraisse limits,
making use of locality.
For this purpose, the 
symmetry properties common to simplicial complexes and the other examples in
Section~\ref{Egs} are of course very relevant.  Thus, for this 
section, let $\K$ be a \genskelish class every member of which satisfies the
Symmetry schema of Definition~\ref{SCdefn}.  
Let $\AlgR$ denote the \Fraisse limit of $\K$,
and let $V=|\AlgR|$; in particular, we drop our assumption from the previous
section that countable structures have underlying set $\N$
(although it would do no harm).

We recall some basic definitions from the theory of group actions,
particularly profinite group actions --- see for example
\cite{Len:PrG} or \cite{Wil:PrG} for background.
Given an action of a group $G$ with identity element $e$ on a set $X$, 
for each $x\in X$ the \emph{stabilizer}
of $x$ is the subgroup of $g\in G$ such that $gx=x$.  We thus say that an
action has \emph{trivial stabilizers} on $Y\subset X$ if for all $y\in Y$, 
$gy=y\implies g=e$.
An action is \emph{faithful} if for all $g\neq h\in G$ there is an $x\in X$
such that $gx\neq hx$.

\begin{lemma} \label{este}
Let $H$ be a finite group, $S \subset V$ a finite subset, and $v \in V \setminus S$ an 
element.  Suppose that $\rho \colon H \times (\AlgR\restr S) \to \AlgR \restr S$ is a left action.  
Then there is a finite subset $S' \subset V$ such that $S \cup \{v\} \subset S'$ and the action 
$\rho $ extends to an action of $H$ on $\AlgR \restr S'$ with trivial stabilizers on 
$S' \setminus S$.
\end{lemma}

\begin{proof}
Enumerate the group $H$ as $h_0=e, h_1,\ldots,h_{n}$.
We choose elements $v_{h_i}\in V\smallsetminus S$ so that 
$\rho$ extends to an action $\bar\rho$ defined on 
$S'=S\cup\{v_h\st h\in H\}$ by 
$\bar\rho(h_i,v_{h_j})=v_{h_ih_j}$ and 
$\bar\rho(h,s)=\rho(h,s)$ for $s\in S$.
We do this recursively in $i\leq n$.
For $i=0$, we choose $v_e = v$. 
For $i>0$, 
let $H_i=\{h_0,\ldots,h_i\}$; we choose the vertex
$v_{h_i}$ such that the following properties hold.
\begin{enumerate}
\item For all $j<i$, $v_{h_i}$ is different from $v_{h_j}$.
\item \label{HrespR}\label{lastindassump}
Suppose that
$R$ is an $(m_1+m_2)$-ary relation in the signature $\Sigma$ of $\K$,
$\mathbf{s}$ is an $m_1$-tuple from $S$, and
$\mathbf{h}$ is an $m_2$-tuple from $H_i$. 
Let
$\mathbf{v}$ be the $m_2$-tuple $(v_g)_{g\in\mathbf{h}}$, and suppose
$h\in H$ is such that for every component $g$ of $\mathbf{h}$,
the product $hg$ is in $H_i$.
Then we require that
$R(\mathbf{s},\mathbf{v})$ holds in $\AlgR$ 
if and only if
$R((\rho(h,s))_{s\in\mathbf{s}},(v_{hg})_{g\in\mathbf{h}})$
does.
\end{enumerate}
Clearly this is exactly what is needed to prove the Lemma with
$S'=S\cup\{v_{h}\st h\in H\}$ and $\bar\rho$ the extension of the action; 
the difficultly is in seeing that achieving
(\ref{HrespR}) is possible.  We now elucidate this.

Suppose $v_{h_0},\ldots,v_{h_{i-1}}$ have been defined 
satisfying conditions (1) and (\ref{lastindassump});
we shall describe how to choose
an appropriate $v_{h_i}$.  To ground the argument, choose an arbitrary
$w\in V\smallsetminus (S\cup\{v_{h_0},\ldots,v_{h_{i-1}}\})$.
Consider the substructure $B$ of $\AlgR$ on 
$|B|=S\cup\{v_{h_0},\ldots,v_{h_{i-1}}\}\cup\{w\}$. 
It lies in $\K$, and $\bar\rho$ (with $w$ as $v_{h_i}$) acts as a partial
action on its underlying set, but probably does not respect its relations.
If we can demonstrate how to change $B$ to another structure $\bar B$ in $\K$
for which $\bar\rho$ respects the relations, 
without changing the substructure on
$S\cup\{v_{h_0},\ldots,v_{h_{i-1}}\}$, 
then weak homogeneity will allow us to choose an appropriate 
$v_{h_i}\in V\smallsetminus(S\cup\{v_{h_0},\ldots,v_{h_{i-1}}\})$
so that $\AlgR\restr S\cup\{v_{h_0},\ldots,v_{h_{i}}\}$ is isomorphic to 
$\bar B$,
whence (\ref{HrespR}) will be satisfied.
With this in mind we temporarily shift our perspective, thinking of
our underlying set $|B|=S\cup\{v_{h_0},\ldots,v_{h_{i-1}},w=v_{h_i}\}$ 
as fixed but the relations satisfied by $B$ as being
mutable, so long as we remain in $\K$ and do not change the induced 
substructure on $S\cup\{v_{h_0},\ldots,v_{h_{i-1}}\}$.

We change our relations $R$ by induction on the arity of $R$; thanks to
General Irreflexivity, 
this process will terminate after $\|B\|$ stages, giving a $\bar B$
for which $\bar\rho$ respects the relations.  So suppose we have a 
$B_m\in\K$ such that for all 
$l<m$ and $R\in L$ of arity $l$, $\bar\rho$ respects $R$, that is,
for all $h\in H$,
$R(t_1,\ldots,t_l)$ if and only if $R(\bar\rho(h,t_1),\ldots\bar\rho(h,t_l))$
whenever all of the $v_g$ terms involved lie in 
$\{v_{h_0},\ldots,v_{h_{i}}\}$.

Enumerate the $m$-element subsets of $|B|$ containing $v_{h_i}$ as
$X_1,\ldots,X_{{\|B\|-1}\choose{m-1}}$.
Similarly to the proof of Lemma~\ref{strAdP},
we proceed to choose the $m$-ary relations that hold for $X_j$ by induction on
$j$, building structures $B_m^j$ with the same $(m-1)$-frame as $B_{m-1}$,
and the relations on $X_k$ for $1\leq k\leq j$ which we shall specify in such
a way that $\bar\rho$ respects $m$-ary relations on these $X_k$ and other
$m$-element subsets of $|B|\smallsetminus\{v_{h_i}\}$ .

For the base case take $B_m^0=B_m$.
Now suppose we have constructed $B_m^{j-1}$.
For each $h\in H$, consider $hX_j=\{\bar\rho(h,x)\st x\in X_j\}$.
If there is an $h$ such that $hX_j\subset(|B|\smallsetminus\{v_{h_i}\}$ or
$hX_j=X_{j'}$ for some $j'<j$, we take $B^j_m$ to be a structure $B'$ as given
by the Adoptive Property for $B=B^{j-1}_m$, $A=B^{j-1}_m\restr X_j$ and 
$A'=B^{j-1}_m\restr hX_j$.  That is, the $m$-ary relations on $X_j$ 
in $B^j_m$ are given by
\[
R(x_1,\ldots,x_m)\iff R(\bar\rho(h,x_1),\ldots,\bar\rho(h,x_m))
\]
whenever $\{x_1,\ldots,x_m\}=X_j$: this is precisely what is
required to make $\bar\rho(h,\cdot)$ respect $R$.  
The requirement on frames for the
Adoptive Property holds by the inductive assumption on $m$ and the fact that
$\fr{m-1}{(B^{j-1}_m)}=\fr{m-1}{B_{m-1}}$.
We also have that $B^j_m$ is well-defined, because if both $hX_j$ and $h'X_j$
are in $[|B|\smallsetminus\{v_{h_i}\}]^m\cup\{X_1,\ldots,X_{j-1}\}$,
then $(h'h^{-1})hX_j=h'X_j$, and so by the inductive assumption on $j$,
the $m$-ary relations on $hX_j$ and $h'X_j$ agree.

Finally, if there is no such $h$, our choice of the $m$-ary relations on
$X_j$ does not matter, and we may take $B^j_{m}=B^{j-1}_{m}$.
This completes the recursive construction of $B^j_m$ over $j$, and hence of
$B_m$ over $m$, and hence of $\bar B\in\K$ with $|\bar B|=|B|$ and
$\bar B\restr(|B|\smallsetminus\{v_{h_i}\})=
B\restr(|B|\smallsetminus\{v_{h_i}\})$
such that $\bar\rho$ respects the relations in $\bar B$.
As described above, the weak homogeneity of $\AlgR$ now lets us choose a
$v_{h_i}\in V$ such that 
$\AlgR\restr S\cup\{v_{h_0},\ldots,v_{h_i}\}$ is isomorphic to $\bar B$,
and we are done.
\end{proof}

\begin{coroll}
Let $G$ be a metrizable profinite group; then $G$ has a continuous faithful 
action on $\AlgR $.
\end{coroll}

\begin{proof}
Recall that a group is a metrizable profinite group if and only if there is a sequence $(H_n)_{n \geq 1}$ of normal open subgroups of $G$ of finite index such that $H_1\supset H_2 \supset \cdots \supset H_n \supset\cdots$ and $G = \mathop{\lim } \limits_{\substack{\longleftarrow \\n}} G/H_n$.
By induction on $n$ we construct a sequence of sets $S_n \subset V$ of vertices of $\AlgR $ such that 
\begin{itemize}
\item $S_n \subset S_{n+1}$;
\item there is a faithful action $\rho _n$ of $G/H_n$ on $\AlgR _{S_n}$;
\item  $\cup_n S_n = V$.
\end{itemize}
Let $V = \{v_i \mid i \in \N\}$ be an enumeration of $V$.  For $n=0$ let $H_0=G$ and $S_0=\{v_0\}$ with the trivial action of the trivial group $G/H_0$.  Suppose that $n \geq 1$ and that we already defined $S_{n-1}$ and a faithful action $\rho _{n-1}$ of $G/H_{n-1}$ on $S_{n-1}$.  Apply Lemma~\ref{este} with $H=G/H_n$ and $S=S_{n-1}$, the action $\rho$ factoring through $\rho _{n-1}$ and $v$ being the vertex of $V \setminus S_{n-1}$ with least index.  We thus obtain a faithful action $\rho _n$ of $G/H_n$ on $S' =: S_n$.

Define an action $\bar \rho $ of $G$ on $\AlgR $ by $\bar \rho (\bar g,v_n) = \rho _n (g,v_n)$, where $g$ is the image of $\bar g$ in $G/H_n$.  The action $\bar \rho$ is faithful, since an element acting trivially on every $v \in V$ is contained in $H_n$ for every $n \geq 1$.  
Because $G=\lim_{\leftarrow}(G/H_n)$, we have $\bigcap H_n = (e)$, 
and are done.
\end{proof}


\begin{lemma} \label{sotto}
Let $G$ be a finite group, $H \subset G$ a subgroup and $S \subset V$ a finite subset.  Suppose that $\rho \colon H \times (\AlgR \restr S) \to \AlgR \restr S$ is a left action with trivial stabilizers.  
Then there is a finite subset $S' \subset V$ with $S' \supset S$ 
and an action $\rho ' \colon G \times \AlgR \restr {S'} \to \AlgR \restr {S'}$ with trivial stabilizers such that $\rho ' |_{H \times \AlgR\,\restr\,S} = \rho$.
\end{lemma}

\begin{proof}
Let $n$ be the number of orbits of $H$ on $S$.  Identify $S$ with its action of $H$ with the subset $H \times \{1,\ldots , n\} \subset G \times \{1,\ldots , n\}$ with the trivial action of $H$ on the second component.  
Exactly as in the proof of Lemma~\ref{este}, we may construct a structure 
$\bar B$ in
$\K$ on underlying set $G\times\{1,\ldots,n\}$ such that the natural 
$G$-action (which extends the $H$-action on $H\times\{1,\ldots,n\}$) respects
the relations of $\bar B$.
Indeed the argument proceeds by induction over $m$ on the $m$-frame, 
and for each $m$ through an induction over $j$ on an enumeration $(X_j)$ of
$[G\times\{1,\ldots,n\}]^m\smallsetminus[H\times\{1,\ldots,n\}]^m$.
By weak homogeneity, the embedding of $H\times\{1,\ldots,n\}\cong S$ into
$\AlgR$ extends to an embedding of $\bar B$ into $\AlgR$, and we may take
$S'$ to be the range of this embedding and $\rho'$ to be the induced 
$G$-action.
\end{proof}

\begin{coroll}\label{injlims}
Let $G$ be the direct limit of a sequence of injections of finite groups $G_n \into G_{n+1}$ for $n \geq 1$.  Then $G$ has an action on $\AlgR $ with trivial stabilizers.
\end{coroll}

\begin{proof}
Let $G_0$ be the trivial group.  We prove by induction that for all $n \geq 0$ there is a subset $S_n \subset V$ with a left action $\rho _n \colon G_n \times \AlgR \restr{S_n} \to \AlgR \restr{S_n}$ with trivial stabilizers such that for all $n \geq 1$ we have 
\begin{itemize}
\item $S_{n-1} \subset S_n$, 
\item $\rho _n |_{G_{n-1} \times \AlgR\,\restr\,{S_{n-1}}} = \rho _{n-1}$ and 
\item $\bigcup S_n = V$.
\end{itemize}
Let $V = \{v_i \mid i \in \N\}$ be an enumeration of $V$.  For $n=0$ let $S_0=\{v_0\}$ and $\rho _0 \colon G_0 \times \AlgR _{S_0} \to \AlgR _{S_0}$ be the unique action.  Suppose that $n \geq 1$ and that $S_{n-1}$ and $\rho _{n-1}$ have been defined; let $v \in V \setminus S_{n-1}$ be the element with least index.  Apply Lemma~\ref{este} to obtain a finite $S' \supset S \cup \{v\}$ and an action $\rho '$ of $G_{n-1}$ on $\AlgR _{S'}$ with trivial stabilizers.  Apply Lemma~\ref{sotto} to $\rho '$ to obtain a subset $S_n \subset V$ and an action $\rho _n \colon G_n \times \AlgR _{S_n} \to \AlgR _{S_n}$ with the required properties.
\end{proof}

This means that for $n>0$, 
subgroups of $GL_n$ over the algebraic closure of a finite field
are subgroups of $\Aut(\AlgR)$.
So too are $S_\infty$ and $\Q/\Z$.
Indeed, in each of these cases, it is clear that every finite subset is 
contained in a finite subgroup.

Corollary~\ref{injlims} allows us to embed many divisible groups into $\Aut(\AlgR)$. 
Note however that $\Aut(\AlgR)$ itself
is not divisible.  To see this, consider the following construction of an element of $G$ of order 4 not admitting
a square root.  Take distinct vertices $v$ and $w$ in $V$
and let a generator $u\in\Z/4\Z$ act on $\{v,w\}$ non-trivially.  Use Lemma~\ref{este} inductively to extend this to a $\Z/4\Z$ action
on $\AlgR$ with trivial stabilizers on $V\setminus\{v,w\}$.  Suppose for contradiction that $g\in\Aut(\AlgR)$
is such that $g^2=u\in\Aut(\AlgR)$. Then $u(g(v))=g(u(v))=g(w)$ and $u(g(w))=g(v)$, and therefore 
$g(v)$ is an element of $V$ with non-trivial stabilizer in $\Z/4\Z$, and thus $g$ permutes $\{v,w\}$.  Hence 
$g^2(v)=v$, contradicting $g^2(v)=u(v)=w$.

\section{Topology}\label{topology}

In this section we establish the homeomorphism type of the random simplicial complex: despite its random construction, the geometric realisation 
of the random simplicial complex is topologically quite simple.

In this section we shall always give a name for the underlying set of a 
simplicial complex, freeing up the notation $|\cdot|$ to represent the
\emph{geometric realisation}, defined as follows.
Let $\SC$ be a simplicial complex on a set $V$.  
Let $e_v$ for $v \in V$ be the standard basis of $\mathbb{R}^V$.  The geometric realisation $|\SC|$ of $\SC$ is the union over faces $F \in \SC$ of the convex hull of the set $\{e_v \st v \in F \}$.  A subset $A \subset |\SC|$ is taken to be open whenever $A \cap \mathbb{R}^W$ is open for all finite $W \subset V$.  
We sometimes identify a vertex $v$ of $\SC $ and the corresponding element $e_v$ in the realisation of $\SC $.

\begin{defn}
Let $n$ be a natural number.  
The \emph{$n$-simplex} $\simplex{n}$ on $\{0,1,\ldots,n\}$ is the simplicial complex consisting of all subsets of $\{0,1,\ldots,n\}$; the 
\emph{standard $n$-simplex} is the geometric realisation of $\simplex{n}$.  
The \emph{infinite simplex} $\simplex{\infty}$ on $\N$ is the simplicial complex consisting of all finite subsets of $\N$, that is $\simplex{\infty}:=\fin{\N}$.
\end{defn}

Let $n$ be an element of $\mathbb{N} \cup \{\infty\}$.  A \emph{piecewise linear map} of a simplicial complex $\SC $ to $\mathbb{R}^n$ is a function from the geometric realisation of $\SC $ to $\mathbb{R}^n$ that is linear on every face of $\SC $.  Thus a piecewise linear function is uniquely determined by its values on the vertices of the geometric realisation of $\SC $.

A simplicial complex $\SC$ is a \emph{cone with vertex $a$} if whenever $F$ is a face of $\SC$ also $F \cup \{a\}$ is a face of $\SC$.

\begin{lemma}\label{subofcone}
Let $n$ be a natural number, 
let $V$ be a set and let $v\in V$.  
Suppose that $\SC $ is a simplicial complex on $V\smallsetminus\{v\}$ 
and that $\SC'$ is a simplicial complex on $V$ 
containing $\SC $ as an induced subcomplex.  
If $\varphi \colon |\SC| \to \mathbb{R}^n$ is a piecewise linear map to $\mathbb{R}^n$, 
then there is a unique piecewise linear map $\varphi' \colon |\SC '| \to \mathbb{R}^{n+1}$ such that $\varphi'(e_v) = (0,\ldots,0,1) \in \mathbb{R}^{n+1}$ and for all $y \in |\SC|$ we have $\varphi' (y) = (\varphi(y),0) \in \mathbb{R}^{n+1}$.  If the map $\varphi$ is a homeomorphism onto its image, then the same is true for the extension $\varphi'$.
\end{lemma}
\begin{proof}
Every point $x$ in $|\SC'|$ is in the convex hull of $\{e_w\st w\in F\}$
for some face $F$ of $\SC'$, and hence may be written as 
$te_v+(1-t)y$ for some $t\in[0,1]$ and $y\in|\SC|$.
It is then straightforward to check that defining 
$\varphi'(x)=((1-t)\varphi(y),t)$ yields a well-defined, piecewise linear
map $\varphi'\colon|\SC'|\to\R^{n+1}$, 
that is a homeomorphism to its image if
$\varphi$ is.
\end{proof}

\begin{lemma} \label{estecono}
Let $V$ be a finite set, let $\SC$ be a simplicial complex on $V$ and let $v$ be a vertex of $\SC$.  
Suppose that $n$ is a natural number and $\varphi \colon |\SC | \to \mathbb{R}^{n+1}$ is a piecewise linear map such that $\varphi (e_v) = (0,\ldots , 0,1) \in \mathbb{R}^{n+1}$ and \(\varphi\restr_{|\SC_{V \setminus \{v\}}|}\) is a homeomorphism to 
the standard $(n-1)$-simplex.  
There exists a finite set $V'$ containing $V$, a simplicial complex $\SC'$ on $V'$ and a piecewise linear map $\varphi' \colon |\SC'| \to \mathbb{R}^{n+1}$ such that 
\begin{itemize}
\item the complex $\SC $ is the induced subcomplex of $\SC'$ on $V$, 
\item the piecewise linear map $\varphi'$ is a homeomorphism to the standard $n$-simplex in $\mathbb{R}^{n+1}$, and 
\item the restriction of $\varphi'$ to $|\SC|$ coincides with $\varphi$.
\end{itemize}
\end{lemma}

\begin{proof}
Let $(F_1 , F_2 , \ldots , F_r)$
be a list of those faces $F$ of $\SC $ such that 
$F \cup \{v\}$ is not a face of $\SC$, ordered by increasing dimension.  
Note in 
particular that for every proper subset $X \subsetneq F_1$ we have that $X \cup \{v\}$ is a face of $\SC $.  
Let $\SC_1$ be the simplicial 
complex obtained by adding to $\SC $ a new vertex $v_1$ 
and having $\{v_1\} \cup G$ as a face if and only if 
$G \subsetneq F_1 \cup \{v\}$.  
(Thus $\SC_1$ is obtained by adding to $\SC$ the cone with apex $v_1$ over the boundary of the simplex $F_1 \cup \{v\}$, and the simplex $F_1 \cup \{v\}$ is not a face of $\SC$.)
We extend $\varphi $ to a piecewise linear map $\varphi _1$ of $|\SC _1|$ into $|\simplex{n}|$ by mapping $v_1$ to the barycentre $b_1$ of 
$\varphi(|F_1 \cup \{v\}|)$.  
Observe that the image of the extension $\varphi_1$ contains the convex hull of $\varphi (F_1 \cup \{v\})$.  

Now suppose that $i \geq 2$ 
and that $\SC_{i-1}$ and $\varphi_{i-1}$ are already defined.  
Let $v_{i_1} , \ldots , v_{i_s}$ be the vertices of 
$\SC_{i-1} \setminus \SC$ that lie on the cone in $\mathbb{R}^{n+1}$ with 
vertex $\varphi(e_v)$ and base the image under $\varphi$ of
the boundary of $|F_i|$; denote 
by $b_{i_1} , \ldots , b_{i_s}$ the images under $\varphi_{i-1}$ of $v_{i_1}, \ldots , v_{i_s}$ respectively.  Let $\SC_i$ be the simplicial 
complex obtained by adding a new vertex $v_i$ to $\SC _{i-1}$ and adding to $\SC _{i-1}$ the cone with vertex $v_i$ on the simplicial 
complex induced by $\SC_{i-1}$ on $F_i \cup \{v\} \cup \{v_{i_1} , \ldots , v_{i_s}\}$.  We extend $\varphi _{i-1}$ to a piecewise linear map 
$\varphi _i \colon |\SC _i| \to|\simplex{n}|$ by letting $\varphi _i (v_i)$ be the 
barycentre $b_i$ of the set consisting of $v$ and the vertices of 
$\varphi_{i-1}(|F_i|)$. 
Observe that the image of the extension $\varphi _i$ contains the convex hulls for $j \leq i$ of 
$\varphi (F_j \cup \{v\})$, 
and that $\varphi _i$ is a homeomorphism onto its image.  
Continuing in this manner,
we conclude that the simplicial complex $\SC':= \SC_r$ and the piecewise linear map $\varphi' := \varphi_r$ satisfy the requirements 
and the proof is complete.
\end{proof}

\begin{thm}\label{toptriv}
The geometric realisation of the random simplicial complex $\RSC$ is 
homeomorphic to the realisation of the infinite simplex $\simplex{\infty}$.
\end{thm}
\begin{proof}
For each $n \in \mathbb{N}$ we construct by induction on $n$ 
a finite induced simplicial complex $\SC_n \subset \RSC$ and a piecewise linear map $\varphi_n \colon |\SC_n| \to \mathbb{R}^{n+1}$ such that 
\begin{itemize}
\item $\SC_n \subset \SC_{n+1}$ and $\cup _n \SC_n = \RSC$; 
\item $\varphi_n$ is a piecewise linear homeomorphism to the standard $n$-simplex in $\mathbb{R}^{n+1}$;
\item the restriction of $\varphi_{n+1}$ to $|\SC_n|$ coincides with $\varphi_n$ followed by the inclusion of $\mathbb{R}^{n+1}$ into $\mathbb{R}^{n+2}$ as the linear subspace with vanishing last coordinate.
\end{itemize}
Fix a bijection between the vertex set of $\RSC$ and $\mathbb{N}$, thus identifying the vertex set of $\RSC$ with $\mathbb{N}$.  For $n=0$ we let $\SC_0$ be the vertex $0 \in \mathbb{N}$ of $\RSC$; thus $|\SC_0|$ is the standard $0$-simplex and we let $\varphi_0$ be the inclusion of $|\SC_0|$ in its geometric realisation in $\mathbb{R}^1$.

Suppose that $\SC_{n-1}$ and $\varphi_{n-1}$ have already been defined and let $v \in \mathbb{N}$ be the least vertex not already appearing in $\SC_{n-1}$.  Let $\SC \subset \RSC$ be the simplicial complex induced on the vertices of $\SC_{n-1}$ and $v$.
Using Lemma~\ref{subofcone},
extend $\varphi_{n-1}$ to a piecewise linear map $\varphi \colon |\SC | \to \mathbb{R}^{n+1}$ by defining $\varphi(v) = (0, \ldots , 0,1)$, and for each vertex $w$ of $\SC_{n-1}$, defining $\varphi(w)=(\varphi_{n-1}(w),0)$.

We are therefore in a position to apply Lemma~\ref{estecono} to 
the simplicial complex $\SC $ and  the piecewise linear map $\varphi$: 
let $\SC'$ and $\varphi'$ be as in that lemma.  
By 
weak homogeneity (see Definition~\ref{wkhom}) of $\RSC$, there is an induced subcomplex $\SC_n$ of $\RSC$ containing $\SC_{n-1}$ and an isomorphism $\iota \colon \SC_n \to \SC'$ such that $\iota$ is the identity on $\SC_{n-1}$.  Let $\varphi_n \colon | \SC_n| \to \mathbb{R}^{n+1}$ be the the piecewise linear map defined by $\varphi_n(w) = \varphi'(\iota(w))$ for each vertex $w \in \SC_n$.  The complex $\SC_n$ and the map $\varphi_n$ satisfy the requirements, and the theorem follows.
\end{proof}

We do not expect
contractibility to be a first order property in our
language.
\begin{qn}\label{C01}
Is it the case that almost all finite simplicial complexes have 
contractible realisation (under the frame-wise uniform measure)?
\end{qn}
It follows from Theorems~\ref{01Law} and \ref{toptriv} that 
a negative answer to Question~\ref{C01} would imply that
contractibility is not a first order property.

\section{Acknowledgement}

We would like to thank Anand Pillay for pointing us to the work of
Koponen~\cite{Kop:APEPRlCS} after seeing a preliminary version of this paper.


\bibliographystyle{plain}
\bibliography{RSC}

\printindex

\end{document}